\documentclass[11pt]{article}
\usepackage{latexsym}

\newcommand{\qed}{\mbox{$\Box$}\vspace{\baselineskip}}

\newenvironment{proof}
{\noindent {\bf Proof:}}{{\qed}}
\newenvironment{proof_special}
{\noindent {\bf Proof:}}{}
\newenvironment{proof_}[1]
{\noindent {\bf#1:}}{{\qed}}

\newtheorem{theorem}{Theorem}[section]
\newtheorem{proposition}[theorem]{Proposition}

\newtheorem{definition}[theorem]{Definition}
\newtheorem{corollary}[theorem]{Corollary}

\newcommand{\av}{{\bf a}}
\newcommand{\bv}{{\bf b}}
\newcommand{\ab}{\av\bv}

\newcommand{\rv}{{\bf r}}

\begin{document}

\title{Enumerative properties of Ferrers graphs}

\author{{\sc Richard Ehrenborg\thanks{Partially
supported by National Science Foundation grant
0200624.}}
           $\:\:$ and $\:\:$
        {\sc Stephanie van
Willigenburg\thanks{Partially supported by the
National Sciences
and Engineering  Research Council of Canada.}}}

\date{{\it To Lou Billera and Andr\'e Joyal on their
           $3 \cdot 4 \cdot 5$th birthdays}}

\maketitle

\begin{abstract}
We define a class of bipartite graphs that correspond
naturally
with Ferrers diagrams. We give expressions for the
number of
spanning trees, the number of Hamiltonian paths when
applicable,
the chromatic polynomial, and the chromatic symmetric
function. We
show that the linear coefficient of the chromatic
polynomial is
given by the excedance set statistic.
\end{abstract}

\section{Introduction and preliminaries}

Geometric and algebraic combinatorics span many areas, from
the geometry of hyperplane
arrangements~\cite{Billera_Brown_Diaconis,Billera_Filliman_Sturmfels},
through graph
theory~\cite{Billera_Chan_Liu,Greene_Zaslavsky}, to
the
more algebraic permutation
statistics~\cite{Billera_Sarangarajan,Ehrenborg_Steingrimsson,Ehrenborg_Steingrimsson_2}.
An important aspect of all these areas is enumeration,
which often
illuminates the finer structure of the object under
investigation, be
it computing the faces of a
polytope~\cite{Bayer_Billera,Billera_Ehrenborg,Billera_Ehrenborg_Readdy}
or the distribution of permutations satisfying certain
criteria~\cite{Billera_Hsiao_van_Willigenburg}.

In this paper we unite these facets of combinatorics via the study
of Ferrers graphs, and in particular answer some of the more
pertinent questions concerning enumeration. More precisely, we
define a class of bipartite graphs that we call Ferrers graphs, so
called since the edges in the graphs are in direct correspondence
with the boxes in a Ferrers diagram.  First, we calculate the
number of spanning trees. The technique we use to prove this
utilizes electrical networks. In fact, the first reference to
spanning trees is in an article by Kirchhoff~\cite{Kirchhoff},
thus the study of trees and the study of electrical networks share
their origin in the work of Kirchhoff. Second, when the two parts
in the vertex partition have the same cardinality we determine the
number of Hamiltonian paths in the Ferrers graph. This result is
based upon the previous result and the proof is inspired by
Joyal's proof of Cayley's formula. Third, and most mysterious, we
prove that the linear coefficient of the chromatic polynomial of a
Ferrers graph is given by the excedance set statistic of
permutations. Lastly, we compute the chromatic symmetric function,
thus generating a family of symmetric functions arising from
Ferrers diagrams other than Schur functions. It should be noted
that our Ferrers graphs are not those appearing
in~\cite{Frumkin_James_Roichman}.

\begin{definition}
Define a \emph{Ferrers graph} to be a bipartite graph on the
vertex partition $U = \{u_{0}, \ldots, u_{n}\}$ and $V = \{v_{0},
\ldots, v_{m}\}$ such that
\begin{itemize}
\item if $(u_{i},v_{j})$ is an edge then so is
$(u_{p},v_{q})$ for $0 \leq p \leq i$ and $0 \leq q
\leq j$.

\item $(u_{0}, v_{m})$ and $(u_{n}, v_{0})$ are edges.
\end{itemize}
\end{definition}

For a Ferrers graph $G$ we have the associated
partition $\lambda =
(\lambda_{0}, \lambda_{1}, \ldots, \lambda_{n})$,
where $\lambda_{i}$
is the degree of the vertex $u_{i}$.  Similarly, we
have the dual
partition $\lambda^{\prime} = (\lambda^{\prime}_{0},
\lambda^{\prime}_{1}, \ldots, \lambda^{\prime}_{m})$,
where
$\lambda^{\prime}_{j}$ is the degree of the vertex
$v_{j}$. The
associated Ferrers diagram is the diagram of boxes
where we have a box
in position $(i,j)$ if and only if $(u_{i},v_{j})$ is
an edge in the
Ferrers graph.

\begin{figure}
\begin{center}
\setlength{\unitlength}{0.7mm}
\begin{picture}(100,40)(0,10)

   \put(20,20){\circle*{2}}
   \put(40,20){\circle*{2}}
   \put(60,20){\circle*{2}}
   \put(80,20){\circle*{2}}
   \put(20,40){\circle*{2}}
   \put(40,40){\circle*{2}}
   \put(60,40){\circle*{2}}

   \put(20,40){\line(0,-1){20}}
   \put(20,40){\line(1,-1){20}}
   \put(20,40){\line(2,-1){40}}
   \put(20,40){\line(3,-1){60}}

   \put(40,40){\line(-1,-1){20}}
   \put(40,40){\line(0,-1){20}}
   \put(40,40){\line(1,-1){20}}
   \put(40,40){\line(2,-1){40}}

   \put(60,40){\line(-2,-1){40}}
   \put(60,40){\line(-1,-1){20}}

   \put(18,44){${u_{0}}$}
   \put(38,44){${u_{1}}$}
   \put(58,44){${u_{2}}$}

   \put(18,14){${v_{0}}$}
   \put(38,14){${v_{1}}$}
   \put(58,14){${v_{2}}$}
   \put(78,14){${v_{3}}$}

\end{picture}
\begin{picture}(100,40)(-30,-5)

   \put(0,0){\line(1,0){20}}
   \put(0,10){\line(1,0){40}}
   \put(0,20){\line(1,0){40}}
   \put(0,30){\line(1,0){40}}

   \put(0,30){\line(0,-1){30}}
   \put(10,30){\line(0,-1){30}}
   \put(20,30){\line(0,-1){30}}
   \put(30,30){\line(0,-1){20}}
   \put(40,30){\line(0,-1){20}}

   \put(13,-5){${\mathbf b}$}
   \put(21,3){${\mathbf a}$}
   \put(24,5){${\mathbf b}$}
   \put(33,5){${\mathbf b}$}
   \put(41,13){${\mathbf a}$}

\end{picture}
\end{center}
\caption{The Ferrers graph and the Ferrers diagram
associated with
the partition $(4,4,2)$, the dual partition
$(3,3,2,2)$ and the
$\ab$-word $\bv \av \bv \bv \av$.}
\label{figure_babba}
\end{figure}
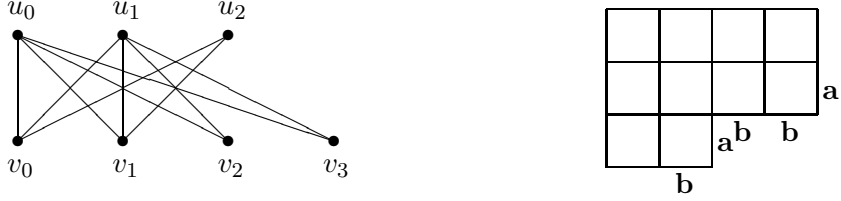

There is another natural way to index Ferrers graphs.
Consider the
Ferrers diagram associated with the graph. Walk along
the path on
the border of the Ferrers diagram starting at the
lower right hand
corner of the box indexed by $(n,0)$ and ending at the
lower right
hand corner of the box indexed by $(0,m)$. Label a
horizontal step
by $\bv$ and a vertical step by $\av$. It is
straightforward to
see that the $\ab$-words obtained this way are in one
to one
correspondence with Ferrers graphs. This is
essentially the same
encoding as in Exercise~7.59 in~\cite{Stanley_EC_II}.
See
Figure~\ref{figure_babba} for an example of a Ferrers
graph, its
Ferrers diagram, partition, dual partition and
$\ab$-word.

\section{The number of spanning trees}
\setcounter{equation}{0}

For a spanning tree $T$ of a Ferrers graph $G$ define
the weight
$\sigma(T)$ to be
$$      \sigma(T)
   =
        \prod_{p=0}^{n} x_{p}^{\deg_{T}(u_{p})}
      \cdot
        \prod_{q=0}^{m} y_{q}^{\deg_{T}(v_{q})}   .
$$
For a Ferrers graph $G$ define $\Sigma(G)$ to be the
sum
$\Sigma(G)
   =
 \sum_{T}   \sigma(T)$,
where $T$ ranges over all spanning trees $T$ of the
Ferrers graph
$G$. Also let $\tau(G)$ denote the number of spanning
trees of the
graph $G$, that is, $\tau(G)
  =
 \Sigma(G)\vrule_{x_{0} = \cdots = x_{n} = y_{0} =
\cdots = y_{m} = 1}$.

\begin{theorem}
Let $G$ be the Ferrers graph corresponding to the
partition
$\lambda$ and the dual partition $\lambda^{\prime}$.
Then the sum
of the weights of spanning trees $T$ of the Ferrers
graph $G$ is
given by
$$  \Sigma(G)
       =
    x_{0} \cdots x_{n} \cdot y_{0} \cdots y_{m} \cdot
    \prod_{p=1}^{n} (y_{0} + \cdots +
y_{\lambda_{p}-1})
  \cdot
    \prod_{q=1}^{m} (x_{0} + \cdots +
x_{\lambda^{\prime}_{q}-1})
 .   $$
Hence the number of spanning trees of $G$ is given by
$$  \tau(G)
       =
    \prod_{p=1}^{n} \lambda_{p}
   \cdot
    \prod_{q=1}^{m} \lambda^{\prime}_{q}   .   $$
\label{theorem_spanning_trees}
\end{theorem}

Using the theory of electrical networks, originating with
Kirchhoff~\cite{Kirchhoff} (for a more accessible reference
see~\cite{Bollobas}), we can deduce the following:

\begin{proposition}
Let $H$ be a Ferrers graph and let $G$ be the Ferrers
graph
obtained from $H$ by adding the edge $(u_{i}, v_{j})$,
where $i,j
\geq 1$. Then the ratio between $\Sigma(G)$ and
$\Sigma(H)$ is
given by
$$   \frac{\Sigma(G)}{\Sigma(H)}
   =
     \frac{x_{0} + \cdots  + x_{i-1} + x_{i}}{x_{0} +
\cdots + x_{i-1}}
   \cdot
     \frac{y_{0} + \cdots  + y_{j-1} + y_{j}}{y_{0} +
\cdots + y_{j-1}}
 .   $$
\label{proposition_Sigma_Sigma}
\end{proposition}
\begin{proof}
Let $N$ be given by $(x_{0} + \cdots + x_{i}) \cdot
                     (y_{0} + \cdots + y_{j})$.
View the Ferrers graph as an electrical network where
the edge
$(u_{p},v_{q})$ is a resistor with resistance
$R(u_{p},v_{q}) =
(x_{p} y_{q})^{-1}$. Assign to each edge in the
Ferrers graph $G$
a current $w(u_{p}, v_{q})$ by the following rule:
$$      w(u_{p}, v_{q})
    =
        \left\{ \begin{array}{c c}
                   -x_{p} y_{q}/N & \mbox{ if } p < i,
q < j, \\
                    y_{q} \sum_{p=0}^{i-1} x_{p} /N &
\mbox{ if } p = i, q <
 j, \\ x_{p} \sum_{q=0}^{j-1} y_{q} /N & \mbox{ if } p
< i, q = j, \\ \left(
                    x_{i} y_{j} +
                    y_{j} \sum_{p=0}^{i-1} x_{p} +
                    x_{i} \sum_{q=0}^{j-1} y_{q}
\right)/N
           & \mbox{ if } p = i, q = j, \\
                         0       & \mbox{ otherwise.}
                \end{array} \right.    $$
Moreover, by Ohm's law we have the potential
difference $P(u_{p},
v_{q}) = R(u_{p},v_{p}) \cdot w(u_{p}, v_{q})$. It is
then
straightforward to verify that $w(u_{p}, v_{q})$ and
$P(u_{p},
v_{q})$ satisfy Kirchhoff's two laws when a current of
size $1$
enters the vertex $u_{i}$ and leaves at $v_{j}$. Also
observe that
the vertices $u_{0}, \ldots, u_{i-1}$ have the same
potential and
hence no current goes through vertices $v_{j+1},
\ldots, v_{m}$.
Similarly, there is no current through the vertices
$u_{i+1},
\ldots, u_{n}$. Hence the current through the edge
$(u_{i},
v_{j})$ is given by
\begin{eqnarray*}
   w(u_{i}, v_{j})
  & = &
\frac{
                    x_{i} y_{j} +
                    y_{j} \sum_{p=0}^{i-1} x_{p} +
                    x_{i} \sum_{q=0}^{j-1} y_{q}
}{N}  \\
  & = &
\frac{N
       -
        \left(\sum_{p=0}^{i-1} x_{p} \right)
           \cdot
        \left(\sum_{q=0}^{j-1} y_{q} \right)}{N}
.
\end{eqnarray*}
However, the current through the edge $(u_{i},v_{j})$
can also be
determined by the theory of electrical networks so
$$ w(u_{i}, v_{j})
       =
   \frac{{\displaystyle \sum_{(u_{i},v_{j}) \in T}
\prod_{e \in T}R(e)^{-1}}} {{\displaystyle \sum_{T}
\prod_{e \in
T} R(e)^{-1}}}
       =
   \frac{\Sigma(G) - \Sigma(H)}{\Sigma(G)}  ,  $$
where the sum in the denominator is over all spanning
trees $T$ of
the Ferrers graph $G$ and the sum in the numerator is
over all
spanning trees containing the edge $(u_{i},v_{j})$. By
combining
the last two identities the result follows.
\end{proof}

\begin{proof_}{Proof of
Theorem~\ref{theorem_spanning_trees}} The proof is by
induction on
the number of edges. The smallest Ferrers graph is the
tree with
$n+m+1$ edges where $(u_{i}, v_{j})$ is an edge if and
only if $i
\cdot j = 0$. This tree has weight $x_{0} \cdots x_{n}
\cdot y_{0}
\cdots y_{m} \cdot x_{0}^{m} \cdot y_{0}^{n}$. The
induction step
adds one edge at a time, and the result follows from
Proposition~\ref{proposition_Sigma_Sigma}.
\end{proof_}

As a corollary of Theorem~\ref{theorem_spanning_trees} we obtain
the classical result for the complete bipartite graphs. For the
history and different approaches of this corollary, see
Exercise~5.30 in~\cite{Stanley_EC_II}.
\begin{corollary}
For the complete bipartite graph $K_{n+1,m+1}$ the sum of the
weights of spanning trees $T$ is given by
$$  \Sigma(K_{n+1,m+1})
       =
    x_{0} \cdots x_{n} \cdot y_{0} \cdots y_{m} \cdot
    (y_{0} + \cdots + y_{m})^{n}
  \cdot
    (x_{0} + \cdots + x_{n})^{m}
 .   $$
Thus the number of spanning trees of $K_{n+1,m+1}$ is
given by
$\tau(K_{n+1,m+1}) = (m+1)^{n} \cdot (n+1)^{m}$.
\end{corollary}

\section{The number of Hamiltonian paths}
\setcounter{equation}{0}

We now turn our attention to enumerating the number of Hamiltonian
(open) paths in a Ferrers graph in the case when $n = m$, that is,
when the two parts in the vertex partition of the bipartite graph
have the same cardinality. Observe that for convenience we will
identify a Hamiltonian path with its reversal.

There are two important structures to consider. The
first one is
vertebrates:
\begin{definition}
Define a {\em vertebrate} $(T,h,t)$ of a Ferrers graph as a
spanning tree $T$ together with one vertex $h$ from the set $U$
called the head and one vertex $t$ from the set $V$ called the
tail. Call the set of vertices on the unique path from the head
$h$ to the tail $t$ the \emph{joints} of the vertebrate.
\end{definition}

Since there are $\lambda^{\prime}_{0}$ ways to choose
a head and
$\lambda_{0}$ ways to choose a tail, we  have as a
direct
corollary to Theorem~\ref{theorem_spanning_trees}:

\begin{corollary}
Let $G$ be the Ferrers graph corresponding to the
partition
$\lambda$ and the dual partition $\lambda^{\prime}$.
Then the
number of vertebrates of the Ferrers graph $G$ is
given by
$$  \prod_{p=0}^{n} \lambda_{p}
   \cdot
    \prod_{q=0}^{m} \lambda^{\prime}_{q}   .   $$
\label{corollary_vertebrate}
\end{corollary}

The other important structure we will work with is permissible
functions on the set $U \cup V$. We call a function $f : U \cup V
\longrightarrow U \cup V$ {\em permissible} if, for all $z \in U
\cup V$,  $(z,f(z))$ is an edge in the associated Ferrers graph.
Observe that the product in Corollary~\ref{corollary_vertebrate}
also enumerates the number of permissible functions on the Ferrers
graph $G$.

For a function $f$ let $f^{k}$ denote the $k$th power
of the
function under composition, that is, $f^{k} = f \circ
\cdots \circ
f$. For a permissible function $f$ call the set $E(f)
= \cap_{k
\geq 1} \mbox{\rm Im}(f^{k})$ the essential set of the
function~$f$. Observe that $f$ restricts to a
permissible
permutation on the set $E(f)$. Moreover, the essential
set $E(f)$
intersects the sets $U$ and $V$ in equally large
subsets.

Using similar ideas of Andr\'e Joyal~\cite{Joyal} we
are able to
prove for Ferrers graphs:

\begin{theorem}
Let $G$ be a Ferrers graph with n=m, that is, each of the two
parts in the vertex partition have the same cardinality. Then the
number of Hamiltonian paths in $G$ is equal to the square of the
number of placements of n+1 rooks on the associated Ferrers board.
\label{theorem_Hamiltonian_paths}
\end{theorem}

Observe that the number of rook placements on a Ferrers board with
$n+1$ rooks is $\lambda_{n} \cdot (\lambda_{n-1} - 1) \cdots
(\lambda_{0} - n)$, where $\lambda$ is the associated partition.
Similarly, this is also equal to $\lambda^{\prime}_{n} \cdot
(\lambda^{\prime}_{n-1} - 1) \cdots (\lambda^{\prime}_{0} - n)$,
where $\lambda^{\prime}$ is the dual partition.

\begin{proof_}{{Proof of
Theorem~\ref{theorem_Hamiltonian_paths}}} First
observe that the
number of rook placements squared is equal to the
number of
permissible bijections $\pi$ on the Ferrers graph $G$.

The proof of the statement is by induction on $n$. The
induction
basis is $n=0$ which is straightforward. Now the
induction step.

Let $S$ be a proper subset of $U \cup V$ such that $S \cap U$ and
$S \cap V$ have equal size. We claim that the number of
vertebrates of the Ferrers graph $G$ with the joints being the set
$S$ is equal to the number of permissible functions on $G$ having
 essential set $S$. By the induction hypothesis we
know that the number of Hamiltonian paths on $G$
restricted to the
set $S$ is equal to the number of permissible
permutations on the
set $S$. Now, a vertebrate is a path such that each
vertex in the
path is the root of a tree. Similarly a function is a
permutation
such that each entry in the permutation is the `root'
of a `tree'.
For instance, for a root $s$ in the essential set $S$
of a
permissible function $f$ the tree is the collection of
vertices
$z$ such that $f^{k}(z) \in S$ implies there exists $i
\leq k$
such that $f^{i}(z) = s$ but $f^j(z)\not\in S$ for
$j<i$. Hence
the claim follows by changing a path on the set $S$ to
a
permissible permutation on the set $S$.

Now by summing over all $S$ strictly contained in $U
\cup V$ we
have that the number of vertebrates that are not paths
is equal to
the number of permissible functions that are not
permutations.
Since the cardinalities of vertebrates and permissible
functions
are the same we are done.
\end{proof_}

\section{The chromatic polynomial and the linear
coefficient}

Before we embark on deriving the chromatic polynomial let us
recall the excedance set statistic. It was first studied
in~\cite{Ehrenborg_Steingrimsson,Ehrenborg_Steingrimsson_2}. We
follow their notation and instead of speaking of the excedance
set, we talk about the excedance word.

Define {\em the excedance word} of a permutation $\pi = \pi_{1}
\cdots \pi_{k+1}$ in $S_{k+1}$ to be the word $w = w_1 \cdots w_k$
where $w_{i} = \av$ if $\pi_{i} \leq i$ and $w_{i} = \bv$ if
$\pi_{i} > i$. For an $\ab$-word $w$ of length $k$ let $[w]$
denote the number of permutations in $S_{k+1}$ with excedance word
$w$.

Following~\cite{Ehrenborg_Steingrimsson} let $R_{m} =
\{\rv=(r_0,\ldots,r_{m}) \: : \: r_{0} = 1, r_{i+1} - r_i \in
\{0,1\}\}$. Thus, each vector $\rv=(r_{0},\ldots,r_{m})$ in
$R_{m}$ starts with $r_{0}=1$ and increases by at most one at each
coordinate. Let $h(\rv)$ be the number of indices $i$ such that
$r_{i+1} = r_{i}$. We then have the following result;
see~\cite[Theorem~6.3]{Ehrenborg_Steingrimsson}.

\begin{theorem}
Let $w$ be an $\av\bv$-word with exactly $m$ $\bv$'s. That is, we
can write $w  = \av^{n_{0}} \bv \av^{n_{1}} \bv \cdots \bv
\av^{n_{m}}$. Then the excedance set statistic $[w]$ is given by
$$
    [w]
       =
    \sum_{\rv \in R_{m}}
       (-1)^{h(\rv)}
         \cdot
       r_{0}^{n_{0} + 1}
         \cdot
       r_{1}^{n_{1} + 1}
         \cdots
       r_{m}^{n_{m} + 1}  .   $$
\label{theorem_inclusion_exclusion}
\end{theorem}

For an $\ab$-word $w$, let $\chi(w)$ denote the chromatic
polynomial in $t$ of the Ferrers graph $G$ associated with  $w$.
Moreover, let $|w|$ denote the length of the $\ab$-word $w$. Now
we can state the relationship between the linear coefficient of
the chromatic polynomial and the excedance set statistic.

\begin{theorem}
The linear coefficient of the chromatic polynomial
$\chi(w)$ is
given by $(-1)^{|w|+1} \cdot [w]$.
\label{theorem_main}
\end{theorem}

It is straightforward to observe that $\chi(\av w) = \chi(w \bv) =
(t-1) \cdot \chi(w)$ and $\chi(1) = t \cdot (t-1)$, where the $1$
in $\chi(1)$ denotes the empty word.

For a vector $\rv$ in the set $R_{m}$ and $1 \leq i
\leq m$ define
$f_{i}(\rv) = f_{i}$ by $f_{i} = t - r_{i-1}$ if
$r_{i} - r_{i-1}
= 1$ and $f_{i} = r_{i-1}$ otherwise.

\begin{theorem}
Let $w$ be an $\av\bv$-word with exactly $m$ $\bv$'s,
that is, $w
= \av^{n_{0}} \bv \av^{n_{1}} \bv \cdots \bv
\av^{n_{m}}$. Then
the chromatic polynomial $\chi(w)$ of the associated
Ferrers graph
$G$ is given by:
$$
    \chi(w)
       =
    \sum_{\rv \in R_{m}}
       t \cdot (t-r_{0})^{n_{0}} \cdot
       f_{1} \cdot (t-r_{1})^{n_{1}} \cdot
       f_{2} \cdots
       f_{m-1} \cdot (t-r_{m-1})^{n_{m-1}} \cdot
       f_{m} \cdot (t-r_{m})^{n_{m} + 1}.
$$
\label{theorem_chromatic}
\end{theorem}
\begin{proof}
For a proper coloring of the graph $G$ let $r_{i}$ be
the number
of distinct colors appearing on the $i+1$ nodes
$v_{0}$ through
$v_{i}$. Let us determine how many colorings there are
of the
graph with a given vector $\rv = (r_{0}, \ldots,
r_{m})$.

The node $v_{0}$ can be colored in $t$ ways. If $r_{i} - r_{i-1} =
1$ then the the node $v_{i}$ is colored with a color not used
before, and there are $t - r_{i-1}$ such colors. If $r_{i+1} -
r_{i} = 0$ then the node is colored with an `old' color, and there
are $r_{i-1}$ such colors. In both cases we have $f_{i}$
possibilities.

For $i \leq m-1$ observe that there are $n_{i}$ $u$-nodes that are
 connected exactly to the nodes $v_{0}, \ldots, v_{i}$. There are
$(t - r_{i})^{n_{i}}$ ways to color these $n_{i}$ nodes, since
they all have to avoid the $r_{i}$ colors of the nodes $v_{0},
\ldots, v_{i}$. Finally, there are $n_{m} + 1$ $u$-nodes that are
connected to all the $v$-nodes $v_{0}, \ldots, v_{m}$. Similarly,
there are $(t - r_{m})^{n_{m}+1}$ ways to color these nodes. Hence
there are
$$
       t \cdot
       f_{1} \cdot
       f_{2} \cdots
       f_{m} \cdot
       (t-r_{0})^{n_{0}} \cdots
       (t-r_{m-1})^{n_{m-1}} \cdot
       (t-r_{m})^{n_{m} + 1}
$$
ways to color the graph $G$ with a given $\rv$-vector.
Now summing
over all possible $\rv$-vectors the result follows.
\end{proof}

We now prove the main result:

\begin{proof_}{{Proof of Theorem~\ref{theorem_main}}}
To obtain the linear coefficient in $\chi(w)$ divide
by $t$ and
set $t=0$. Observe that $f_{i}$ evaluated at $t=0$ is
equal to
$r_{i-1}$ with a sign change if $r_{i} - r_{i-1} = 1$.
The number
of such sign changes is $m-h(\rv)$. Moreover we also
obtain $n_{0}
+ n_{1} + \cdots + n_{m} + 1$ sign changes from the
other factors.
Hence the total number of sign changes is $m-h(\rv) +
n_{0} +
n_{1} + \cdots + n_{m} + 1
  =|w| - h(\rv) + 1$.

The remainder of the term corresponding to $\rv$ can
now be
written as $r_{0}^{n_{0} + 1} \cdot
 r_{1}^{n_{1} + 1} \cdots
 r_{m}^{n_{m} + 1}$, and the result follows by
Theorem~\ref{theorem_inclusion_exclusion}.
\end{proof_}

There is one important special case of
Theorem~\ref{theorem_chromatic}:

\begin{proposition}
The chromatic polynomial of the complete bipartite
graph
$K_{n+1,m+1}$ is given by
$$
    \chi(\bv^{m} \av^{n})
  =
    \sum_{k=1}^{m+1}
        S(m+1,k) \cdot t \cdot (t-1) \cdots (t-k+1)
\cdot (t-k)^{n+1}
    ,
$$
where $S(m,k)$ denotes the Stirling number of the
second kind.
\label{proposition_b_m_a_n}
\end{proposition}
\begin{proof}
Begin to color the vertices $v_{0}, v_{1}, \ldots, v_{m}$ with
exactly $k$ colors where $1 \leq k \leq m+1$. This can be done in
$S(m+1,k) \cdot t \cdot (t-1) \cdots (t-k+1)$ ways. There are
$(t-k)^{n+1}$ ways to color the remaining vertices $u_{0}, u_{1},
\ldots, u_{n}$.
\end{proof}

The linear coefficient of the chromatic polynomial (up
to a sign)
also has the interpretation of being the number of
acyclic
orientations of the graph with a {\em unique} given
sink~\cite{Greene_Zaslavsky}. Also observe that it is
enough to
note that there are no directed $4$-cycles in an
orientation of
the edges in a Ferrers graph to guarantee that the
orientation is
acyclic. Expressing this in terms of the associated
Ferrers
diagram we have:
\begin{corollary}
The excedance set statistic $[w]$ is the number
colorings of the
boxes in the Ferrers diagram associated to the
$\ab$-word $w$ with
colors red and blue such that
\vspace*{-3 mm}
\begin{itemize}
\item[(i)] there are no four boxes $(p,r), (p,s), (q,r),
(q,s)$ such that
$(p,r)$ and $(q,s)$ are colored red and $(p,s)$ and
$(q,r)$ are
colored blue,

\item[(ii)] there is a unique given row where all the boxes
are colored
red, and

\item[(iii)] there is no column where all the boxes are
colored blue.
\end{itemize}
\label{corollary_coloring}
\end{corollary}

\section{The chromatic symmetric function}

A natural generalization of the chromatic polynomial,
known as the
\emph{chromatic symmetric function} was defined
in~\cite{Stanley_Chromatic}, and it is natural to ask
whether we can
explicitly compute these for Ferrers graphs. This
would give us a
set of symmetric functions other than the Schur
functions that can
be computed from Ferrers diagrams.

Observe that unlike the Schur functions, the chromatic
symmetric
functions of Ferrers graphs will not form a basis for
the
symmetric functions as the chromatic symmetric
function of the
Ferrers graph corresponding to the partition $\lambda$
and
$\lambda '$ will be identical.

Before we continue we need to define the constitution
of a Ferrers
diagram whose boxes have been colored red and blue.
First choose a
red box. Score through that row and column. For every
red box with
a score going through it in one direction score
through it in the
other direction. Repeat until all the red boxes either
have two
scores or no scores through them. Extract all the
boxes with two
scores in them. Choose another red box, and repeat
until none
remain. The list of extractions is the
\emph{constitution} and
each extraction is called a \emph{constituent}.

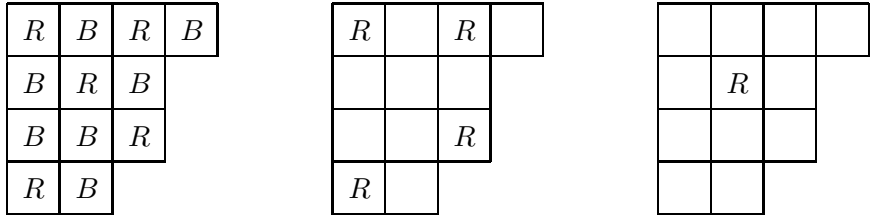
\begin{figure}
\begin{center}
\setlength{\unitlength}{0.7mm}
\begin{picture}(60,40)(-10,-5)

   \put(0,0){\line(1,0){20}}
   \put(0,10){\line(1,0){30}}
   \put(0,20){\line(1,0){30}}
   \put(0,30){\line(1,0){40}}
   \put(0,40){\line(1,0){40}}

   \put(0,40){\line(0,-1){40}}
   \put(10,40){\line(0,-1){40}}
   \put(20,40){\line(0,-1){40}}
   \put(30,40){\line(0,-1){30}}
   \put(40,40){\line(0,-1){10}}

   \put(3,3){$R$}
   \put(13,3){$B$}

   \put(3,13){$B$}
   \put(13,13){$B$}
   \put(23,13){$R$}

   \put(3,23){$B$}
   \put(13,23){$R$}
   \put(23,23){$B$}

   \put(3,33){$R$}
   \put(13,33){$B$}
   \put(23,33){$R$}
   \put(33,33){$B$}

\end{picture}
\begin{picture}(60,40)(-10,-5)

   \put(0,0){\line(1,0){20}}
   \put(0,10){\line(1,0){30}}
   \put(0,20){\line(1,0){30}}
   \put(0,30){\line(1,0){40}}
   \put(0,40){\line(1,0){40}}

   \put(0,40){\line(0,-1){40}}
   \put(10,40){\line(0,-1){40}}
   \put(20,40){\line(0,-1){40}}
   \put(30,40){\line(0,-1){30}}
   \put(40,40){\line(0,-1){10}}

   \put(3,3){$R$}

   \put(23,13){$R$}

   \put(3,33){$R$}
   \put(23,33){$R$}

\end{picture}
\begin{picture}(60,40)(-10,-5)

   \put(0,0){\line(1,0){20}}
   \put(0,10){\line(1,0){30}}
   \put(0,20){\line(1,0){30}}
   \put(0,30){\line(1,0){40}}
   \put(0,40){\line(1,0){40}}

   \put(0,40){\line(0,-1){40}}
   \put(10,40){\line(0,-1){40}}
   \put(20,40){\line(0,-1){40}}
   \put(30,40){\line(0,-1){30}}
   \put(40,40){\line(0,-1){10}}

   \put(13,23){$R$}

\end{picture}
\end{center}
\caption{A Ferrers diagram with colored boxes, and its
constituents.} \label{figure_two}
\end{figure}

In addition, let $RB_{\lambda}$ be the set of all red-blue
colorings  of the Ferrers diagram corresponding to the partition
$\lambda$ (without the restriction of
Corollary~\ref{corollary_coloring}).
For $r\in RB_\lambda$ let $|r|$ be the
number of constituents of $r$ and $|r|_{red}$ be the number of
boxes in $r$ colored red.

\begin{theorem}\label{chromsymm}
Let $G$ be the Ferrers graph corresponding to the
partition
$\lambda$. Then the chromatic symmetric function $X_G$
in terms of
the power sum symmetric functions $p_{\mu}$ is given
by:
$$X_G=\sum _{r\in RB_\lambda} (-1)^{|r|_{red}}
        \cdot
    p_{r_1}\cdot p_{r_2} \cdots p_{r_{|r|}}
\cdot p_{1}^b   ,  $$ where $r_i$ is the number of
rows plus the
number of columns in the $i$th constituent of $r$,
$1\leq i \leq
|r|$ and $b$ is the number of rows plus the number of
columns of
$r$ that contain no red boxes.
\end{theorem}

\begin{proof}
Recall that for a graph $G$ with a set of edges $E$ the definition
of the chromatic symmetric function in terms of the power sum
basis is~\cite[Theorem 2.5]{Stanley_Chromatic}
$$X_G=\sum _{S\subseteq E} (-1)^{|S|}
p_{|C_0|} \cdots p_{|C_m|}$$ where $|C_i|$ is the number of
vertices in each connected component $C_i, 0\leq i\leq m$ of $G$
with the edges not in $S$ removed.

Now observe that for a Ferrers graph $G$ with edge set
$E$ there
is an natural bijection between $S\subseteq E$ and
red-blue
colorings $r\in RB_\lambda$ of the Ferrers diagram
associated with
$\lambda$,   given by
$$(u_i,v_j)\in S \Leftrightarrow (i,j) \mbox{ is
colored red in }r.$$ This gives us the index of
summation and the
exponent of $-1$ in our formula. To complete the proof
note the
constituents of $r$ yield precisely the connected
components of
$G$ containing more than one vertex, and if the $i$th
row (column)
of $r$ contains only blue boxes then $u_i$ ($v_i$) is
not
connected to any other vertex in $G$.
\end{proof}

A more specific formula can be found for the two
extreme cases of
Ferrers graphs. First the case when the Ferrers graph
is a tree.

\begin{corollary}\label{hook_chromsymm}
Let $G$ be the Ferrers graph corresponding to the
partition
$(m+1)1^n$. Then the chromatic symmetric function
$X_G$ in terms
of the power sum symmetric functions $p_{\mu}$ is
given by:
$$
  X_G
   =
  \sum _{i=0}^{m+n}
      (-1)^i\left( \left(\sum_{j+k=i} {m\choose j}{n
\choose k} p_{j+1} p_{k+1} p_{1}^{m+n-i}\right)
-{m+n\choose i}
p_{i+2} p_{1}^{m+n-i}\right).$$
\end{corollary}

\begin{proof_special}
Observe that in the case where the Ferrers diagram
associated with
$\lambda$ is a hook, for $r\in RB_\lambda$ if $(0,0)$
is blue then
we obtain the the function
$$\sum _{i=0}^{m+n}  (-1)^i\sum_{j+k=i} {m\choose j}{n
\choose k} p_{j+1} p_{k+1} p_{1}^{m+n-i}$$
whereas if it is red then
we obtain the function
$$
\hspace*{5 mm} \hspace*{48 mm} \sum _{i=1}^{m+n+1}
(-1)^i {m+n
\choose i-1} p_{i+1} p_{1}^{m+n+1-i}. \hspace*{48 mm}
\qed $$
\end{proof_special}

The other extreme case is the complete bipartite graph
$K_{n,m}$,
which is the Ferrers graph associated with the
partition $m^{n}$.
A change of basis is required for the simplest
description of the
chromatic symmetric function.

\begin{corollary}
The chromatic symmetric function $X_{K_{n,m}}$
in terms of
the monomial symmetric functions $m_{\mu}$ is given
by:
$$ X_{K_{n,m}}
  =
 \sum_{\sigma \in \Pi_{n}}
 \sum_{\tau   \in \Pi_{m}}
       (r_1!r_2! \cdots )
     \cdot
       m_{\mu (\sigma , \tau)}    ,  $$
where $\Pi_{n}$ is the collection of all set
partitions of
$\{1, \ldots, n\}$,
$\mu(\sigma ,\tau)$ is the partition determined by the
block sizes
of $\sigma$ and $\tau$, and $r_i$ is the multiplicity
of $i$ in
$\mu(\sigma, \tau)$.
\label{Kmn_chromsymm}
\end{corollary}
\begin{proof}
Recall that a stable partition of the vertices of a
graph $G$ is a
partition of the vertices such that each block is
totally
disconnected.
Then Proposition~2.4 in~\cite{Stanley_Chromatic}
states
$$ X_G = \sum_{\pi}
            \left( r_1!r_2! \cdots \right)
             m_{\mu (\pi)}             ,
$$
where the sum ranges over all stable partitions $\pi$
of the graph
$G$. The result follows by noting that in the complete
bipartite
graph $K_{n,m}$, every block in a stable partition
either lies
entirely in the $n$ vertices $\{u_0, \ldots,
u_{n-1}\}$ or lies
entirely in the $m$ vertices $\{v_0, \ldots,
v_{m-1}\}$.
\end{proof}

The symmetric functions appearing in
Corollary~\ref{Kmn_chromsymm}
have the following explicit exponential generating
function,
generalizing Exercise~5.6 in~\cite{Stanley_EC_II}:
$$    \sum_{n,m \geq 0} X_{K_{n,m}}
              \frac{s^{n}}{n!} \frac{t^{m}}{m!}
   =
       \prod_{i \geq 1} \left( e^{s x_i} + e^{t x_i} -
1 \right) ,  $$ where we view the symmetric functions
in terms of
the variables $\{x_i\}_{i \geq 1}$.

Lastly, note that to recover the earlier chromatic polynomial we
set $x_1=\ldots = x_t=1$ and all other $x_i=0$.

\section{Concluding remarks}

Is it possible to obtain an expression for the Tutte polynomial of a
Ferrers graph, that would both encode the number of spanning trees in
Theorem~\ref{theorem_spanning_trees} and the chromatic polynomial in
Theorem~\ref{theorem_chromatic}? For the enumerative results in this
paper it is natural to ask for combinatorial proofs. From a bijection
given in~\cite{Remmel_Williamson}, a bijective proof for
Theorem~\ref{theorem_spanning_trees} can be obtained via some
modifications. In~\cite{Burns} bijective proofs for
Theorems~\ref{theorem_spanning_trees}
and~\ref{theorem_Hamiltonian_paths} have been derived using box
labeling. However, it would also be desirable to have a bijective
proof for Corollary~\ref{corollary_coloring}.

The excedance set statistic $[w]$ satisfies the
recursion
$[u\bv\av v]
   =
 [u \av\bv v]
   +
 [u \av v]
   +
 [u \bv v]$
where $u$ and $v$ are two $\ab$-words. Is there a similar
recursion for the chromatic polynomial? A partial answer to this
question is the following proposition, whose proof we omit.
\begin{proposition}
The chromatic polynomial $\chi(w)$ of the associated
Ferrers graph
satisfies the recursion:
$$
\chi(w \bv \av^{k-1})
  =
t \cdot \chi(w \av^{k-1}) +
      \sum_{0 \leq i \leq k-1}
         (-1)^{k-i} \cdot {k \choose i} \cdot \chi(w
\av^{i})   .
$$
\end{proposition}
On the excedance statistic level this recursion
corresponds to $[w
\bv \av^{k-1}]
  =
  \sum_{0 \leq i \leq k-1}
    {k \choose i} \cdot [w \av^{i}]$;
see~\cite[Proposition~2.5]{Ehrenborg_Steingrimsson}.
Moreover, can
this proposition be extended to the chromatic
symmetric function?

Another question related to the chromatic polynomial
arises from
the following observation. A Ferrers graph $G$ can be
equivalently
viewed as an $(n+m+2)$-dimensional hyperplane
arrangement given by
$$
      x_{i} = y_{j}
              \:\:\:\:
          \mbox{ if and only if }
              \:\:\:\:
          \mbox{ $(u_{i},v_{j})$ is an edge in $G$. }
$$
Thus the chromatic polynomial of the Ferrers graph $G$
is also the
characteristic polynomial of the associated hyperplane
arrangement, see~\cite{Greene_Zaslavsky}. Hence, can a
combinatorial expression be found for the number of
acyclic
orientations of the Ferrers graph, or equivalently for
the number
of regions of the associated hyperplane arrangement?

Finally, one can define the Ferrers graph associated
with a skew
partition $\lambda/\mu$. Do any of the results in this
paper
extend naturally to skew partitions?

\section*{Acknowledgments}

The authors thank Margaret Readdy, Tom Zaslavsky and the two
referees for their comments on earlier drafts of this paper.

\newcommand{\journal}[6]{{\sc #1,} #2, {\it #3} {\bf
#4} (#5), #6.}
\newcommand{\preprint}[3]{{\sc #1,} #2, preprint #3.}
\newcommand{\book}[4]{{\sc #1,} ``#2,'' #3, #4.}
\newcommand{\JCTA}{J.\ Combin.\ Theory Ser.\ A}
\newcommand{\JCTB}{J.\ Combin.\ Theory Ser.\ B}
\newcommand{\AdvancesinMathematics}{Adv.\ Math.}
\newcommand{\JournalofAlgebraicCombinatorics}{J.\ Algebraic Combin.}

\noindent {\small {\em  R.\ Ehrenborg,
      Department of Mathematics,
      University of Kentucky,
      Lexington, KY 40506-0027,
      USA. Email: jrge@ms.uky.edu
} }

\noindent {\small {\em  S.\ van Willigenburg,
      Department of Mathematics,
      University of British Columbia,
      1984 Mathematics Road,
      Vancouver, BC, V6T 1Z2,
      Canada. Email: steph@math.ubc.ca
} }

\end{document}